\documentclass[10pt, conference]{IEEEtran}
\usepackage[english]{babel }

\usepackage[margin=54pt]{geometry}

\title{\vspace*{18pt} Distributed Secondary Frequency Control through MTDC Transmission Systems}

\author{ \IEEEauthorblockA{Martin Andreasson$^{\IEEEauthorrefmark{1}}$, Roger Wiget, Dimos V. Dimarogonas, Karl H. Johansson and G\"oran Andersson
}
\\
\thanks{M. Andreasson, D. V. Dimarogonas and K. H. Johansson are with the ACCESS Linnaeus Centre, KTH Royal Institute of Technology, Stockholm, Sweden. R. Wiget and G. Andersson are with the Power Systems Laboratory, ETH Zurich, Zurich, Switzerland. This work was supported in part by the European Commission, the Swedish Research Council (VR) and the Knut and Alice Wallenberg Foundation.
\IEEEauthorrefmark{1} Corresponding author (mandreas@kth.se).}
} 
\IEEEoverridecommandlockouts

\usepackage{mathrsfs}
\usepackage{amsfonts}
\usepackage{amssymb}
\usepackage{amsmath}
\usepackage{graphicx}
\usepackage{amsthm}
\usepackage{commath}
\usepackage{tikz}
\usepackage{tkz-graph}
\usepackage[europeanresistors]{circuitikz}
\usetikzlibrary{arrows}
\usepackage{flushend}
\usepackage{psfrag}
\usepackage{pgfplots}
\usepackage{todonotes}
\usetikzlibrary{external}
\usepackage{blockdiagram}
\usepackage{caption}
\usepackage{subcaption}
\usepackage{color}
\usepackage{booktabs}
\newcommand{\raa}[1]{\renewcommand{\arraystretch}{#1}}
\usepgfplotslibrary{external}
\newtheorem{theorem}{Theorem}

\newtheorem{remark}{Remark}

\newtheorem{assumption}{Assumption}
\newtheorem{objective}{Objective}

\DeclareMathOperator*{\diag}{diag}

\newcommand{\beq}{\begin{equation}}
\newcommand{\eeq}{\end{equation}}
\newcommand{\bq}{\begin{eqnarray}}
\newcommand{\eq}{\end{eqnarray}}
\newcommand{\bqn}{\begin{eqnarray*}}
\newcommand{\eqn}{\end{eqnarray*}}
\newcommand{\bee}{\begin{enumerate}}
\newcommand{\eee}{\end{enumerate}}

\begingroup
\makeatletter
\renewcommand{\p@subfigure}{}
\makeatother

\newlength\fheight
\newlength\fwidth

\begin{document}
\maketitle

\begin{abstract}
In this paper, we present distributed controllers for sharing primary and secondary frequency control reserves for asynchronous AC transmission systems, which are connected through a multi-terminal HVDC grid.
By using Lyapunov arguments, the equilibria of the closed-loop system are shown to be globally asymptotically stable. We quantify the static errors of the voltages and frequencies, and give upper bounds for these errors. It is also shown that the controllers have the property of power sharing, i.e., primary and secondary frequency control reserves are shared fairly amongst the AC systems. The proposed controllers are applied to a high-order dynamic model of of a power system consisting of asynchronous AC grids connected through a six-terminal HVDC grid.
\end{abstract}

\section{Introduction}

Transmitting power over long distances with minimal losses is one of the greatest challenges in today's power transmission systems. The strong rising share of renewables increased the distances between power generation and consumption. This is a driving factor behind the development of long-distance power transmission technologies. One such example are large-scale off-shore wind farms, which often require power to be transmitted in cables over long distances to the mainland power grid \cite{breseti2007HVDC}. Due to the high resistive losses in AC cables, high-voltage direct current (HVDC) power transmission is a commonly used technology for power transmission in these cases. The higher investment cost of an HVDC transmission system compared to an AC transmission system, is compensated by the lower resistive losses for sufficiently long distances \cite{melhem2013electricity}. The break-even point, i.e., the point where the total construction and operation costs of overhead HVDC and AC lines are equal, is typically 500--800 km \cite{padiyar1990hvdc}. However, for cables, the break-even point is typically less than 50 km \cite{Hertem2010technical}. Increased use of HVDC technologies for electrical power transmission suggests that future HVDC transmission systems are likely to consist of multiple terminals connected by several HVDC transmission lines \cite{Haileselassie2013Power}. Such systems are referred to as Multi-terminal HVDC (MTDC) systems in the literature.

Maintaining an adequate DC voltage is the single most important practical control problem for HVDC transmission systems. If the DC voltage deviates too far from the nominal operational voltage, equipment could be damaged, resulting in loss of power transmission capability and high costs.

Many existing AC grids are connected through HVDC links, which are typically used for bulk power transfer between AC areas. The fast operation of the DC converters, however, also enables frequency regulation of one of the connected AC grids through the HVDC link. One practical example of this is the island of Gotland in Sweden, which is only connected to the Nordic grid through an HVDC cable \cite{axelsson2001gotland}. However, since the Nordic grid has orders of magnitudes higher inertia than the AC grid of Gotland, the influence of the frequency regulation on the Nordic grid is negligible.
By connecting several AC grids by an MTDC system, primary frequency regulation reserves may be shared, which reduces the need for frequency regulation reserves in the individual AC systems \cite{li2008frequency}. In \cite{dai2010impact}, distributed control algorithms have been applied to share primary frequency control reserves of asynchronous AC transmission systems connected through an MTDC system. However, the proposed controller requires a slack bus to control the DC voltage, defeating the purpose of distributing the primary frequency regulation reserves. In \cite{andreasson2014MTDCjournal}, distributed controllers for secondary voltage control of MTDC systems are proposed, which do not rely on a slack bus. The analysis in the aforementioned reference is however restricted to the dynamics of the MTDC system, thus neglecting any dynamics of connected AC systems.

In \cite{dai2011voltage} and \cite{silva2012provision}, decentralized controllers are employed to share primary frequency control reserves. In \cite{silva2012provision} no stability analysis of the closed-loop system is performed, whereas \cite{dai2011voltage} guarantees stability provided that the connected AC areas have identical parameters and the voltage dynamics of the HVDC system are neglected. In \cite{taylordecentralized}, optimal decentralized controllers for AC systems connected by HVDC systems are derived. In contrast to the aforementioned references, \cite{andreasson2014distributed} also considers the dynamics of connected AC systems as well as the dynamics of the MTDC system.
A distributed controller, relying on a communication network, is proposed in \cite{dai2013voltage}. Stability is guaranteed in the absence of communication delays. The voltage dynamics of the MTDC system are however neglected. Moreover the implementation of the controller requires every controller to access measurements of the DC voltages of all terminals.

In this paper we present two distributed controllers for combined primary and secondary frequency control of asynchronous AC systems connected through an MTDC system. The first controller requires that the communication network constitutes a complete graph, while the second controller only relies on local information from neighboring AC systems. In contrast to existing controllers in the literature, the proposed controllers in this paper consider the voltage dynamics of the MTDC system, in addition to the frequency dynamics of the AC systems. The equilibrium of the closed-loop system is shown to be globally asymptotically stable for any set of controller parameters. We furthermore bound the asymptotic errors of the voltages and frequencies at the equilibrium, and show the achievable performance is better than for the corresponding decentralized controller studied in \cite{andreasson2014distributed}. We also show that the frequency control reserves are asymptotically shared approximately equally, which is referred to as power sharing in the literature.

 The remainder of this paper is organized as follows. In Section \ref{sec:prel}, the mathematical notation is defined. In Section \ref{sec:model}, the system model and the control objectives are defined. In Section \ref{sec:dec_control}, we recall the decentralized proportional controller for distributing primary frequency control proposed in \cite{andreasson2014distributed}. In Section \ref{sec:secondary_frequency_control}, two secondary frequency controllers for sharing primary and secondary frequency control reserves are presented.
 In Section \ref{sec:simulations}, simulations of the distributed controller on a six-terminal MTDC test system are provided, showing the effectiveness of the proposed controller. The paper ends with concluding remarks in Section \ref{sec:discussion}.

\section{Preliminaries}
\label{sec:prel}
Let $\mathcal{G}$ be a static, undirected graph. Denote by $\mathcal{V}=\{ 1,\hdots, n \}$ the vertex set of $\mathcal{G}$, and by $\mathcal{E}=\{ 1,\hdots, m \}$ the edge set of $\mathcal{G}$. Let $\mathcal{N}_i$ be the set of neighboring vertices to $i \in \mathcal{V}$.
Denote by $\mathcal{B}$ the vertex-edge incidence matrix of $\mathcal{G}$, and let $\mathcal{\mathcal{L}_W}=\mathcal{B}W\mathcal{B}^T$ be the weighted Laplacian matrix of $\mathcal{G}$, with edge-weights given by the  elements of the diagonal matrix $W$. We denote the space of real-valued $n\times m$-valued matrices by $\mathbb{R}^{n\times m}$.
Let $\mathbb{C}^-$ denote the open left half complex plane, and $\bar{\mathbb{C}}^-$ its closure. We denote by $c_{n\times m}$ a vector or matrix of dimension $n\times m$, whose elements are all equal to $c$. For a symmetric matrix $A$, $A>0 \;(A\ge 0)$ is used to denote that $A$ is positive (semi) definite. $I_{n}$ denotes the identity matrix of dimension $n$. For simplicity, we will often drop the notion of time dependence of variables, i.e., $x(t)$ will be denoted $x$. Let $\norm{\cdot}_\infty$ denote the maximal absolute value of the elements of a vector.

\section{Model and problem setup}
\label{sec:model}
We will here give a unified model for an MTDC system interconnected with several asynchronous AC systems.
We consider an MTDC transmission system consisting of $n$ converters, each connecting to an AC system, denoted $1, \dots, n$. The converters are assumed to be connected by an MTDC transmission grid. The dynamics of converter $i$ is assumed to be given by
\begin{align}
\begin{aligned}
C_i \dot{V}_i &= -\sum_{j\in \mathcal{N}_i} \frac{1}{R_{ij}}(V_i -V_j) + I_i^{\text{inj}} ,
\end{aligned}
\label{eq:voltage}
\end{align}
where $V_i$ is the voltage of converter $i$, $C_i>0$ is its capacitance, and $I_i^{\text{inj}}$ is the injected current from an AC grid connected to the DC converter.  The constant $R_{ij}$ denotes the resistance of the HVDC transmission line connecting the converters $i$ and $j$. 
The graph corresponding to the HVDC line connections is assumed to be connected.
 The AC system is assumed to consist of a single generator which is connected to the corresponding DC converter, representing an aggregate model of the AC grid. The dynamics of the AC system are given by the swing equation \cite{machowski2008power}:
\begin{align}
m_i \dot{\omega}_i &=  P^\text{gen}_i + P_i^\text{nom} + P_i^{{m}} - P_i^{\text{inj}}, \label{eq:frequency}
\end{align}
where $m_i>0$ is its moment of inertia. The constant $P_i^\text{nom}$ is the nominal generated power, $P^m_i$ is the uncontrolled deviation from the nominal generated power, $P_i^{\text{inj}}$ is the power injected to the DC system through the convertera, $P^\text{gen}_i$ is the generated power by the generation control (primary or secondary) of generator $i$, respectively.
The control objective can now be stated as follows.
\begin{objective}
\label{obj:1}
The generation control action should be asymptotically distributed fairly amongst the generators, i.e.,
\begin{align*}
\lim_{t\rightarrow \infty } \left|  P_i^{\text{gen}}(t) + \frac{1}{n} \sum_{i=1}^n P_i^m  \right| \le  e^{\text{gen}} \quad \forall i = 1, \dots, n,
\end{align*}
where $e^{\text{gen}}$ is a given scalar.
Furthermore, the frequencies of the AC systems, as well as the converter voltages, should not deviate too far from their nominal values, i.e.,
\begin{align*}
\lim_{t\rightarrow \infty} |V_i(t)-V_i^{\text{ref}}| &\le e^{{V} } \quad \forall i = 1, \dots, n \\
\lim_{t\rightarrow \infty} |\omega_i(t)-\omega^{\text{ref}}| &\le e^{{\omega} } \quad \forall i = 1, \dots, n, \\
\end{align*}
where $V_i^{\text{ref}}$ is the reference DC voltage of converter $i$, $\omega^{\text{ref}}$ is the reference frequency and $e^{{V}}$ and $e^{{\omega} }$ are given scalars.
\end{objective}
\begin{remark}
It is in general not possible to have $e^V=0$, since this does not allow for the  currents in the HVDC grid to change. The bounds $e^\text{gen}$ and $e^\omega$ on the other hand, can in theory be zero. However, this may be hard to achieve with limited communication.
 \end{remark}

\section{Decentralized MTDC control}
\label{sec:dec_control}
In this section we summarize previous results on decentralized control of MTDC systems. The detailed results and proofs can be found in \cite{andreasson2014distributed}.
We consider the following decentralized frequency droop controllers for the control of the AC systems connected through an MTDC network
\begin{align}
P^\text{gen}_i=-K_i^{\text{droop}} (\omega_i-\omega^{\text{ref}}),
\label{eq:droop_control}
\end{align}
 where $\omega_i$ is the frequency of the generator, $\omega^{\text{ref}}$ is the reference frequency (e.g., $50$ or $60$ Hz), and $K_i^{\text{droop}}>0$ .
The local controllers governing the power injections into the MTDC network are given by
\begin{align}
\label{eq:voltage_control}
\begin{aligned}
P_i^{\text{inj}} = P_i^{\text{inj, nom}} + K_i^{{\omega}} (\omega_i - \omega^{\text{ref}}) + K_i^{{V}}(V_i^{\text{ref}}-V_i),
\end{aligned}
\end{align}
where $P_i^{\text{inj, nom}}$ is the nominal injected power, and $K_i^{{\omega}}>0$ and $K_i^{{V}}>0$ are positive controller gains for all $i=1, \dots, n$. The HVDC converter is assumed to be perfect and instantaneous, i.e., injected power on the AC side is immediately converted to DC power without losses. Furthermore the dynamics of the converter are ignored, implying that the converter tracks the output of controller \eqref{eq:voltage_control} perfectly. This assumption is reasonable due to the dynamics of the converter typically being orders of magnitudes faster than the AC dynamics \cite{kundur1994power}.
The relation between the injected HVDC current and the injected AC power is thus given by
\begin{align}
V_iI_i^{\text{inj}} = P_i^{\text{inj}}. \label{eq:power-current_nonlinear}
\end{align}
By assuming $V_i=V^{\text{nom}}\; \forall i=1, \dots, n$, we obtain
\begin{align}
V^{\text{nom}}I_i^{\text{inj}} = P_i^{\text{inj}}. \label{eq:power-current}
\end{align}
Furthermore, we assume that the nominal generated power equals to the nominal injected power.
\begin{assumption}
\label{ass:balances_power}
$P^\text{nom}_i=P^\text{inj, nom}_i \; \forall i=1, \dots, n$.
\end{assumption}

\begin{theorem}
\label{th:stability_passivity_1}
The equilibrium of the decentralized MTDC control system with dynamics given by \eqref{eq:voltage}, \eqref{eq:frequency}, \eqref{eq:droop_control}, \eqref{eq:voltage_control}, and where \eqref{eq:power-current} holds, is globally asymptotically stable \cite{andreasson2014distributed}.
\end{theorem}
We will now study the asymptotic voltages and frequencies as well as the generated power of the MTDC system.
 We make the following assumption on the controller gains.
\begin{assumption}
\label{ass:scalar_1} The controller gains satisfy
$K^\omega_i=k^\omega, K^\text{droop}_i=k^\text{droop}, K^V_i=k^V \; \forall i=1, \dots, n$.
\end{assumption}
\begin{theorem}
\label{th:equilibrium}
Given that Assumption \ref{ass:scalar_1} holds, then for the HVDC and AC systems \eqref{eq:voltage}, \eqref{eq:frequency} with generation control \eqref{eq:droop_control} and converter control \eqref{eq:voltage_control} and where \eqref{eq:power-current} holds, Objective \ref{obj:1} is satisfied for $e^{\text{gen}} = e^{\text{gen}}_\text{dec}$, $e^V=e^V_\text{dec}$ and $e^\omega=e^\omega_\text{dec}$ \cite{andreasson2014distributed}, where
\begin{align*}
e^{\text{gen}}_\text{dec} &=\frac{k^\text{droop}\max_i P^m_i}{k^\text{droop}+k^\omega} \left( (n-1) + \frac{k^V}{V^\text{nom}} \sum_{i=2}^n \frac{1}{\lambda_i(\mathcal{L}_R)} \right) \\
e^V_\text{dec} &= \frac{k^\omega}{nk^\text{droop}k^V}  \left| \sum_{i=1}^n P^m_i \right| {+}  \frac{k^\omega \max_i \left| P^m_i  \right| }{(k^\omega {+} k^\text{droop})V^\text{nom}} \sum_{i=2}^n \frac{1}{\lambda_i(\mathcal{L}_R)} \\
e^\omega_\text{dec} &= \frac{1}{n k^\text{droop}} \left| \sum_{i=1}^n P^m_i \right| \\
&\;\;\;\; + \frac{\max_i P^m_i}{k^\text{droop}+k^\omega} \left( (n-1) + \frac{k^V}{V^\text{nom}} \sum_{i=2}^n \frac{1}{\lambda_i(\mathcal{L}_R)} \right).
\end{align*}
\end{theorem}
In the following section we will show that the above upper bounds can be tightened by a secondary control layer.

\section{Distributed secondary frequency control}
\label{sec:secondary_frequency_control}
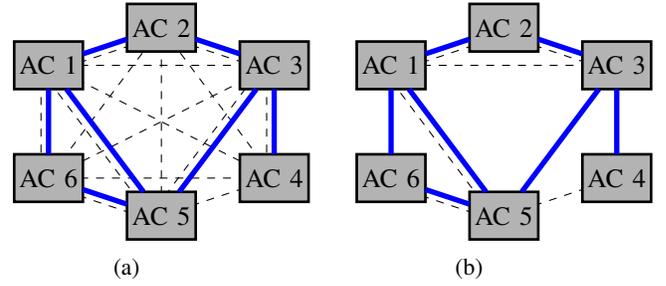
\begin{figure}
\begin{subfigure}[b]{0.17\textwidth}
\begin{tikzpicture}
\GraphInit[vstyle=Normal]
\SetVertexNormal[Shape = rectangle,%
LineWidth = 1pt,%
FillColor = black!30]

\draw (-0.5,0) node {} -- (2.5,0) [dashed] node {};
\draw (-0.5,-0.1) node {} -- (1,0.4) [dashed] node {};
\draw (2.5,-0.1) node {} -- (1,0.4) [dashed] node {};
\draw (-0.5,-0.2) node {} -- (1,-2.2) [dashed] node {};
\draw (1,-2.1) node {} -- (-0.5,-1.6) [dashed] node {};
\draw (1,-2.0)  node {} -- (2.5,-1.5) [dashed] node {};
\draw (-0.6,-0.0)  node {} -- (-0.6, -1.5) [dashed] node {};
\draw (2.4,-0.0)  node {} -- (2.4, -1.5) [dashed] node {};
\draw (1,0.5)  node {} -- (1, -2) [dashed] node {};
\draw (1,0.5)  node {} -- (2.5, -1.5) [dashed] node {};
\draw (1,0.5)  node {} -- (-0.5, -1.5) [dashed] node {};
\draw (-0.5,0) node {} -- (2.5, -1.5) [dashed] node {};
\draw (2.5,0) node {} -- (-0.5, -1.5) [dashed] node {};
\draw (2.5,0.2) node {} -- (1, -1.8) [dashed] node {};
\draw (-0.5,-1.5) node {} -- (2.5, -1.5) [dashed] node {};

\Vertex[x=-0.5, y=0] {AC 1}
\Vertex[x=1, y=0.5] {AC 2}
\Vertex[x=2.5, y=0] {AC 3}
\Vertex[x=-0.5, y=-1.5] {AC 6}
\Vertex[x=2.5, y=-1.5] {AC 4}
\Vertex[x=1, y=-2] {AC 5}

\SetUpEdge[color=blue]
\renewcommand*{\EdgeLineWidth}{2pt}
\Edge[](AC 1)(AC 2)
\Edge[](AC 2)(AC 3)
\Edge[](AC 3)(AC 4)
\Edge[](AC 3)(AC 5)
\Edge[](AC 6)(AC 5)
\Edge[](AC 1)(AC 6)
\Edge[](AC 1)(AC 5)
\end{tikzpicture}
\caption{}
\label{fig:HVDC.cen}
\end{subfigure}
\qquad \qquad
\begin{subfigure}[b]{0.17\textwidth}
\begin{tikzpicture}
\GraphInit[vstyle=Normal]
\SetVertexNormal[Shape = rectangle,%
LineWidth = 1pt,%
FillColor = black!30]

\draw (-0.5,0) node {} -- (2.5,0) [dashed] node {};
\draw (-0.5,-0.1) node {} -- (1,0.4) [dashed] node {};
\draw (2.5,-0.1) node {} -- (1,0.4) [dashed] node {};
\draw (-0.5,-0.2) node {} -- (1,-2.2) [dashed] node {};
\draw (1,-2.1) node {} -- (-0.5,-1.6) [dashed] node {};
\draw (1,-2.0)  node {} -- (2.5,-1.5) [dashed] node {};

\Vertex[x=-0.5, y=0] {AC 1}
\Vertex[x=1, y=0.5] {AC 2}
\Vertex[x=2.5, y=0] {AC 3}
\Vertex[x=-0.5, y=-1.5] {AC 6}
\Vertex[x=2.5, y=-1.5] {AC 4}
\Vertex[x=1, y=-2] {AC 5}

\SetUpEdge[color=blue]
\renewcommand*{\EdgeLineWidth}{2pt}
\Edge[](AC 1)(AC 2)
\Edge[](AC 2)(AC 3)
\Edge[](AC 3)(AC 4)
\Edge[](AC 3)(AC 5)
\Edge[](AC 6)(AC 5)
\Edge[](AC 1)(AC 6)
\Edge[](AC 1)(AC 5)

\end{tikzpicture}
\caption{}
\label{fig:HVDC.dist}
\end{subfigure}
\caption{Illustration of the HVDC grid and the communication network topologies. The HVDC lines are illustrated with solid lines, while comminication lineas are illustrated with dashed lines. \eqref{fig:HVDC.cen} shows the structure of the complete communication structure of  \eqref{eq:droop_control_secondary}, while \eqref{fig:HVDC.dist} shows the structure of the communication structure of  \eqref{eq:droop_control_secondary_distributed}.}
\label{fig:HVDC}
\end{figure}
In this section we study a distributed secondary frequency controller to tighten the stationary error bounds of the closed-loop system given in Theorem~\ref{th:equilibrium}. To implement a distributed controller, we assume the existence of a communication, enabling communication between all AC generators. We will consider two distributed secondary controllers; a controller where  any pair of generators can communicate directly, and a controller where only neighbouring generators can communicate directly. These two controllers correspond require a complete communication graph and a communication graph containing a spanning tree, respectively. The corresponding topologies of these controllers are illustrated in Figure~\ref{fig:HVDC}.
The first generation controller of the AC systems is given by
\begin{align}
P^\text{gen}_i &=- K_i^{\text{droop}} (\omega_i-\omega^{\text{ref}}) -  \frac{K^V_i}{K^\omega_i} K^\text{droop, I}_i \frac{1}{n} \sum_{i=1}^n  \eta_i \nonumber \\
\dot{\eta}_i &= K_i^{\text{droop,I}}(\omega_i-\omega^{\text{ref}}) - \gamma \eta_i,
\label{eq:droop_control_secondary}
\end{align}
where $\gamma, K^\text{droop, I}_i, \; i=1, \dots, n$ are positive constants. If $\gamma=0$, then $\eta_i$ becomes a scaled integral state of the local frequency deviation $(\omega_i-\omega^{\text{ref}})$. 
The second generation controller of the AC systems is given by
\begin{align}
P^\text{gen}_i &=- K_i^{\text{droop}} (\omega_i-\omega^{\text{ref}}) -  \frac{K^V_i}{K^\omega_i} K^\text{droop, I}_i \eta_i \nonumber \\
\dot{\eta}_i &= K_i^{\text{droop,I}}(\omega_i-\omega^{\text{ref}}) - \delta \sum_{j\in \mathcal{N}_i} c_{ij} (\eta_i-\eta_j).
\label{eq:droop_control_secondary_distributed}
\end{align}
We assume $c_{ij} = c_{ji}$, i.e., the communication graph is undirected.
The above controller can be interpreted as a distributed PI-controller, with a distributed consensus filter acting on the integral states $\eta_i$. In contrast to the controller \eqref{eq:droop_control_secondary}, \eqref{eq:droop_control_secondary_distributed} does not rely on a complete communication graph. Note that we make no assumption that the communication topology resembles the topology of the MTDC system.

\subsection{Stability analysis}
\label{sec:secondary_frequency_control_stability}
Combining the voltage dynamics \eqref{eq:voltage}, the frequency dynamics \eqref{eq:frequency} and the generation control \eqref{eq:droop_control_secondary} or \eqref{eq:droop_control_secondary_distributed}, the converter controller \eqref{eq:voltage_control} and the power-current relationship \eqref{eq:power-current}, together with Assumption~\ref{ass:balances_power} and defining $\eta = [\eta_1, \dots, \eta_n]^T$, $\hat{\omega}=\omega - \omega^\text{ref}1_n$ and $\hat{V}=V - V^\text{ref}$, where $V^\text{ref} = [V_1^\text{ref}, \dots, V_n^\text{ref}]$, we obtain the closed-loop dynamics given by \eqref{eq:cl_dynamics_vec_delta_int} or \eqref{eq:cl_dynamics_vec_delta_int_distributed}, respectively where 
 $M=\diag({m_1}^{-1}, \hdots , {m_n}^{-1})$ is a matrix of inverse generator inertia, 
 $E=\diag([C_1^{-1}, \dots, C_n^{-1}])$ is a matrix of electrical elastances, 
 $K^{\text{droop,I}} = \diag([K^{\text{droop,I}}_1, \dots , K^{\text{droop,I}}_n])$, and $\mathcal{L}_c$ is the weighted Laplacian matrix of the communication graph with edge-weights $c_{ij}$. The details of the derivations have been omitted, but the derivation follows the steps taken in Section~IV in \cite{andreasson2014distributed}.
\begin{figure*}
\begin{align}
\begin{bmatrix}
\dot{\hat{\omega}} \\ \dot{\hat{V}} \\ \dot{\eta}
\end{bmatrix}
&= {\begin{bmatrix}
-M(K^\omega + K^{\text{droop}}) & MK^V & - \frac{1}{n} M K^V (K^\omega)^{-1} K^{\text{droop,I}} 1_{n\times n} \\
\frac{1}{V^{\text{nom}}}EK^\omega & -E\left(\mathcal{L}_R + \frac{K^V}{V^{\text{nom}}} \right) & 0_{n\times n}\\
K^{\text{droop,I}} & 0_{n\times n} & -\gamma I_n
\end{bmatrix}}
\begin{bmatrix}
\hat{\omega} \\ \hat{V} \\ \eta
\end{bmatrix} +
\begin{bmatrix}
M P^m \\
0_{n}\\
0_{n}
\end{bmatrix} \label{eq:cl_dynamics_vec_delta_int} \\
\begin{bmatrix}
\dot{\hat{\omega}} \\ \dot{\hat{V}} \\ \dot{\eta}'
\end{bmatrix}
&= {\begin{bmatrix}
-M(K^\omega + K^{\text{droop}}) & MK^V & -  M K^V (K^\omega)^{-1} K^{\text{droop,I}} 1_n \\
\frac{1}{V^{\text{nom}}}EK^\omega & -E\left(\mathcal{L}_R + \frac{K^V}{V^{\text{nom}}} \right) & 0_n\\
\frac{1}{n} 1_n^T K^{\text{droop,I}} & 0_n^T & -\gamma
\end{bmatrix}}
\begin{bmatrix}
\hat{\omega} \\ \hat{V} \\ \eta'
\end{bmatrix} +
\begin{bmatrix}
M P^m \\
0_{n}\\
0
\end{bmatrix} \label{eq:cl_dynamics_vec_delta_int_projected} \\
\begin{bmatrix}
\dot{\hat{\omega}} \\ \dot{\hat{V}} \\ \dot{\eta}
\end{bmatrix}
&= {\begin{bmatrix}
-M(K^\omega + K^{\text{droop}}) & MK^V & - M K^V (K^\omega)^{-1} K^{\text{droop,I}} \\
\frac{1}{V^{\text{nom}}}EK^\omega & -E\left(\mathcal{L}_R + \frac{K^V}{V^{\text{nom}}} \right) & 0_{n\times n}\\
K^{\text{droop,I}} & 0_{n\times n} & -\delta \mathcal{L}_c
\end{bmatrix}}
\begin{bmatrix}
\hat{\omega} \\ \hat{V} \\ \eta
\end{bmatrix} +
\begin{bmatrix}
M P^m \\
0_{n}\\
0_{n}
\end{bmatrix} \label{eq:cl_dynamics_vec_delta_int_distributed}
\end{align}
\noindent\makebox[\linewidth]{\rule{\textwidth}{0.4pt}}
\end{figure*}
Clearly the representation \eqref{eq:cl_dynamics_vec_delta_int} is not minimal with respect to the output $y=[\hat{\omega}^T, \hat{V}^T]^T$, since the integral states $\eta$ are redundant. It is easily shown that substituting the projection
\begin{align}
\eta' &= \frac{1}{n} 1_n^T \eta, \label{eq:eta_transformation}
\end{align}
in \eqref{eq:cl_dynamics_vec_delta_int}, the output dynamics remain unchanged. Even though the dynamics \eqref{eq:cl_dynamics_vec_delta_int} are redundant, they show that the secondary control layer can be implemented distributively. Each AC generator needs only to compute its local integral state, and communicate this to the remaining generators.
In order to prove stability of the closed-loop system, it is however essential to consider the reduced system after applying the projection \eqref{eq:eta_transformation}, which is given by \eqref{eq:cl_dynamics_vec_delta_int_projected}. This system can be interpreted as a centralized implementation of \eqref{eq:cl_dynamics_vec_delta_int}, where the single integral state $\eta'$ is computed centrally with access to all frequency measurements.

Assume that the system matrices of \eqref{eq:cl_dynamics_vec_delta_int_projected} and \eqref{eq:cl_dynamics_vec_delta_int_distributed} are full-rank, which ensures that unique equilibria of \eqref{eq:cl_dynamics_vec_delta_int_projected} and \eqref{eq:cl_dynamics_vec_delta_int_distributed}  exist. Denote these equilibria $x_{0,1}=[\omega_{0,1}^T, V_{0,1}^T, \eta_{0,1}^T]^T$ and $x_{0,2}=[\omega_{0,2}^T, V_{0,2}^T, \eta_{0,2}^T]^T$, respectively. Define $\bar{x}_1\triangleq [\bar{\omega}_1^T, \bar{V}_1^T, \bar{\eta}_1^T]^T =[\hat{\omega}^T, \hat{V}^T, \eta^T]^T - [\omega_{0,1}^T, V_{0,1}^T, \eta_{0,1}^T]^T$ and $\bar{x}_2$, mutatis mutandis. Now: 
\begin{align}
\dot{\bar{x}}_1 = A \bar{x}_1 \label{eq:dynamics_A_distributed_shifted_1} \\
\dot{\bar{x}}_2 = A \bar{x}_2
\label{eq:dynamics_A_distributed_shifted_2}
\end{align}
with the origin as the unique equilibria of both above dynamical systems. We are now ready to show the main stability result of this section.

\begin{theorem}
\label{th:stability_passivity_2}
The equilibria of the systems defined by \eqref{eq:cl_dynamics_vec_delta_int_projected} and \eqref{eq:cl_dynamics_vec_delta_int_distributed} are globally asymptotically stable.
\end{theorem}
\begin{proof}
First consider the Lyapunov function candidate 
\begin{align}
W(\bar{\omega}, \bar{V}, \bar{\eta}') &= \frac 12 \bar{\omega}^T K^\omega (K^V)^{-1} M^{-1}\bar{\omega} + \frac{V^\text{nom}}{2} \bar{V}^T E \bar{V} \nonumber \\
&\;\;\;\; + \frac 12(\bar{\eta}')^2 . \label{eq:lyap_hvdc_secondary_projected}
\end{align}
Clearly $W(\bar{\omega}, \bar{V}, \bar{\eta}')$ is positive definite and radially unbounded. Differentiating \eqref{eq:lyap_hvdc_secondary_projected} with respect to time along trajectories of \eqref{eq:dynamics_A_distributed_shifted_1}, we obtain
\begin{align*}
&\dot{W}(\bar{\omega}, \bar{V}, \bar{\eta}')  \\
&= \bar{\omega}^T K^\omega (K^V)^{-1} M^{-1}\dot{\bar{\omega}} + V^\text{nom} \bar{V}^T E \dot{\bar{V}} + \bar{\eta}' \dot{\bar{\eta}}' \\
&= \bar{\omega}^T \big( -K^\omega (K^V)^{-1}(K^\omega + K^\text{droop})\bar{\omega} \\
&\;\;\;\; + K^\omega \bar{V} - \frac 1n K^\text{droop, I} 1_n \bar{\eta}' \big) \\
&\;\;\;\; + \bar{V}^T \Big( K^\omega \bar{\omega} - (V^\text{nom}\mathcal{L}_R {+} K^V)\bar{V} \Big) \\
&\;\;\;\; + \bar{\eta}'^T \Big( \frac 1n 1_n^T K^\text{droop, I} \bar{\omega}' - \gamma \bar{\eta}' \Big) \\
&= -\bar{\omega}^T \big( -K^\omega (K^V)^{-1}(K^\omega + K^\text{droop})\bar{\omega} \\
&\;\;\;\; +  2 \bar{\omega}^T K^\omega \bar{V} - \bar{V}^T (V^\text{nom}\mathcal{L}_R + K^V)\bar{V}  - \gamma (\bar{\eta}')^2  \\
&= - \begin{bmatrix}
\bar{\omega}^T & \bar{V}^T 
\end{bmatrix}
\underbrace{\begin{bmatrix}
K^\omega (K^V)^{-1}(K^\omega + K^\text{droop}) & -K^\omega \\
-K^\omega & K^V 
\end{bmatrix}}_{\triangleq Q_1}
\begin{bmatrix}
\bar{\omega} \\ \bar{V}
\end{bmatrix}  \\
&\;\;\;\; - \gamma (\bar{\eta}')^2. 
\end{align*} 
Clearly $\dot{W}(\bar{\omega}, \bar{V}, \bar{\eta}')\le 0$ iff the symmetric matrix $Q_1$ is positive definite. By applying the Schur complement condition for positive definiteness to $Q_1$, we see that $Q_1$ is positive definite iff
\begin{eqnarray*}
K^\omega (K^V)^{-1}(K^\omega + K^\text{droop}) - K^\omega (K_V)^{-1} K^\omega \\
= K^\omega (K^V)^{-1} K^\text{droop} > 0.
\end{eqnarray*}
Hence $Q_1$ is always positive definite. If $\gamma > 0$, $\dot{W}(\bar{\omega}, \bar{V}, \bar{\eta}')< 0$, and the origin of  \eqref{eq:dynamics_A_distributed_shifted_1} is thus globally asymptotically stable. If however $\gamma=0$ then $\dot{W}(\bar{\omega}, \bar{V}, \bar{\eta}')< 0$ and the set where $W(\bar{\omega}, \bar{V}, \bar{\eta}')$ is non-decreasing is given by
\begin{align*}
G &= \{ (\bar{\omega}, \bar{V}, \bar{\eta}') | \dot{W}(\bar{\omega}, \bar{V}, \bar{\eta}') =0 \} \\ 
&=  \{ (\bar{\omega}, \bar{V}, \bar{\eta}') | \bar{\eta}' = k \},
\end{align*}
for any $k\in \mathbb{R}$. Clearly the largest invariant set in $G$ is the origin. Thus, by LaSalle's theorem for global stability, the origin  of  \eqref{eq:dynamics_A_distributed_shifted_1} is globally asymptotically stable also for $\gamma=0$. 

Consider now the following Lyapunov function candidate 
\begin{align}
W(\bar{\omega}, \bar{V}, \bar{\eta}) &= \frac 12 \bar{\omega}^T K^\omega (K^V)^{-1} M^{-1}\bar{\omega} + \frac{V^\text{nom}}{2} \bar{V}^T C \bar{V} \nonumber \\
&\;\;\;\; + \frac 12 \bar{\eta}^T \bar{\eta},  \label{eq:lyap_hvdc_secondary}
\end{align}
where  $C=\diag([C_1, \dots, C_n])$. 
Clearly $W(\bar{\omega}, \bar{V}, \bar{\eta})$ is positive definite and radially unbounded. Differentiating \eqref{eq:lyap_hvdc_secondary} with respect to time along trajectories of \eqref{eq:dynamics_A_distributed_shifted_2}, we obtain
\begin{align*}
&\dot{W}(\bar{\omega}, \bar{V}, \bar{\eta})  \\
&= \bar{\omega}^T K^\omega (K^V)^{-1} M^{-1}\dot{\bar{\omega}} + V^\text{nom} \bar{V}^T E \dot{\bar{V}} + \bar{\eta}^T \dot{\bar{\eta}} \\
&= \bar{\omega}^T \big( -K^\omega (K^V)^{-1}(K^\omega + K^\text{droop})\bar{\omega} \\
&\;\;\;\; + K^\omega \bar{V} - K^\text{droop, I} \bar{\eta} \big) \\
&\;\;\;\; + \bar{V}^T \Big( K^\omega \bar{\omega} - (V^\text{nom}\mathcal{L}_R {+} K^V)\bar{V} \Big) \\
&\;\;\;\; + \bar{\eta}^T \Big( K^\text{droop, I} \bar{\omega} - \mathcal{L}_\eta  \Big) \\
&= -\bar{\omega}^T \big( -K^\omega (K^V)^{-1}(K^\omega + K^\text{droop})\bar{\omega} \\
&\;\;\;\; +  2 \bar{\omega}^T K^\omega \bar{V} - \bar{V}^T (V^\text{nom}\mathcal{L}_R + K^V)\bar{V}  - \bar{\eta}^T \mathcal{L}_\eta \bar{\eta}  \\
&= - \begin{bmatrix}
\bar{\omega}^T & \bar{V}^T 
\end{bmatrix}
\underbrace{\begin{bmatrix}
K^\omega (K^V)^{-1}(K^\omega + K^\text{droop}) & -K^\omega \\
-K^\omega & K^V 
\end{bmatrix}}_{\triangleq Q_1}
\begin{bmatrix}
\bar{\omega} \\ \bar{V}
\end{bmatrix}  \\
&\;\;\;\; - \bar{\eta}^T \mathcal{L}_\eta \bar{\eta} \le 0, 
\end{align*} 
since $Q_1$ is positive definite. The set where $W(\bar{\omega}, \bar{V}, \bar{\eta})$ is non-decreasing is given by
\begin{align*}
G &= \{ (\bar{\omega}, \bar{V}, \bar{\eta}) | \dot{W}(\bar{\omega}, \bar{V}, \bar{\eta}) =0 \} \\ 
&=  \{ (\bar{\omega}, \bar{V}, \bar{\eta}) | \bar{\eta} = k 1_n \},
\end{align*}
for any $k\in \mathbb{R}$. Clearly the largest invariant set in $G$ is the origin. Thus, by LaSalle's theorem for global stability, the origin  of  \eqref{eq:dynamics_A_distributed_shifted_2} is globally asymptotically stable.
\end{proof}

\subsection{Equilibrium analysis}
\label{sec:secondary_frequency_control_equilibrium}
In this section we study the properties of the equilibria of \eqref{eq:cl_dynamics_vec_delta_int_projected} and \eqref{eq:cl_dynamics_vec_delta_int_distributed}, respectively. We show that by employing the aforementioned secondary frequency control schemes, it is possible to tighten the error bounds in Theorem~\ref{th:equilibrium}. Analogous to Section~\ref{sec:dec_control}, we assume uniform controller gains.

\begin{theorem}
\label{th:equilibrium_secondary}
Assume that Assumptions  \ref{ass:balances_power} and \ref{ass:scalar_1} hold, mutatis mutandis. Consider the HVDC and AC systems \eqref{eq:voltage}, \eqref{eq:frequency} with generation control \eqref{eq:droop_control} and where the relation \eqref{eq:power-current} holds. Objective \ref{obj:1} is satisfied for the secondary controller \eqref{eq:droop_control_secondary} when $\gamma \rightarrow 0^+$, and for the secondary controller \eqref{eq:droop_control_secondary_distributed} when $\delta \rightarrow + \infty$. In both cases, Objective \ref{obj:1} is satisfied for $e^{\text{dist}} = e^{\text{gen}}_\text{dist}$, $e^V=e^V_\text{dec}$ and  $e^\omega=e^\omega_\text{dist}$, where
\begin{align*}
e^{\text{gen}}_\text{dist} &=\frac{k^\text{droop}\max_i P^m_i}{k^\text{droop}+k^\omega} \left( (n-1) + \frac{k^V}{V^\text{nom}} \sum_{i=2}^n \frac{1}{\lambda_i(\mathcal{L}_R)} \right) \\
e^V_\text{dist} &=   \frac{k^\omega \max_i \left| P^m_i  \right| }{(k^\text{droop}+k^\omega)V^\text{nom}} \sum_{i=2}^n \frac{1}{\lambda_i(\mathcal{L}_R)} \\
e^\omega_\text{dist} &= \frac{\max_i P^m_i}{k^\text{droop}+k^\omega} \left( (n-1) + \frac{k^V}{V^\text{nom}} \sum_{i=2}^n \frac{1}{\lambda_i(\mathcal{L}_R)} \right) .
\end{align*}
Furthermore $1_n^T\hat{\omega}  = 1_n^T\hat{V} = 0$, i.e., the average frequency and voltage deviations are zero.
\end{theorem}

\begin{proof}
Before studying the equilibria of \eqref{eq:cl_dynamics_vec_delta_int_projected} and \eqref{eq:cl_dynamics_vec_delta_int_distributed}, we will show that under the assumptions that $\gamma \rightarrow 0^+$, and $\delta \rightarrow + \infty$, the first $2n$ rows of the equilibria of \eqref{eq:cl_dynamics_vec_delta_int_projected} and \eqref{eq:cl_dynamics_vec_delta_int_distributed} are identical.
Consider the last row of the equilibrium of \eqref{eq:cl_dynamics_vec_delta_int_projected}, which as $\gamma\rightarrow 0$ implies $1_n^T\hat{\omega} = 0$. Thus, premultiplying the $(n+1)$th to $2n$th rows of the equilibrium of \eqref{eq:cl_dynamics_vec_delta_int_projected} with $1_n^TC$ yields $1_n^T\hat{V} = 0$. Finally, premultiplying the first $n$ rows of \eqref{eq:cl_dynamics_vec_delta_int_projected} with $1_n^TM^{-1}$ yields
\begin{align}
\frac{n k^V k^\text{droop, I}}{k^\omega} \eta' = - 1_n^TP^m.  \label{eq:cl_dynamics_vec_delta_int_projected_equilibrium_gamma}
\end{align}
Now consider the equilibrium of \eqref{eq:cl_dynamics_vec_delta_int_distributed}.
Following similar steps as in the manipulation of \eqref{eq:cl_dynamics_vec_delta_int_projected}, we obtain that $1_n^T\hat{\omega} = 0$ and $1_n^T\hat{V} = 0$. Thus premultiplying the first $n$ rows of the equilibrium of \eqref{eq:cl_dynamics_vec_delta_int_distributed} with $1_n^TM^{-1}$, we obtain
\begin{align}
\frac{k^V k^{\text{droop, I}}}{k^\omega} 1_n^T \eta &= - 1_n^T P^m.
\label{eq:cl_dynamics_vec_delta_int_distributed_eq_first_n}
\end{align}
We write $\eta$ as a linear combination
$\eta = \sum_{i=1}^n a^0_i v^0_i$,
where $v^0_i$ is the (normed) $i$th eigenvector of $\mathcal{L}_c$. Inserting the eigen decomposition of $\eta$ in \eqref{eq:cl_dynamics_vec_delta_int_distributed_eq_first_n} yields
\begin{align*}
a^0_i &= - \frac{1_n^T I^\text{inj} k^\omega \sqrt{n}}{k^V k^{\text{droop, I}}},
\end{align*}
since $v^0_1 = \frac{1}{\sqrt{n}} 1_n$. In order to determine $a_i^0$ for $i\ge 2$, we consider the last $n$ rows of the equilibrium of \eqref{eq:cl_dynamics_vec_delta_int_distributed}. Again, using the eigen decomposition of $\eta$ we obtain
\begin{align*}
k^{\text{droop, I}} \hat{\omega} &= \delta \mathcal{L}_c \sum_{i=1}^n a^0_i v^0_i =  \sum_{i=2}^n a^0_i \lambda^0_i v^0_i.
\end{align*}
Premultiplying the above equation with $(v^0_j)^T$ we obtain
\begin{align*}
a^0_j = \frac{k^{\text{droop, I}} (v^0_j)^T \hat{\omega}}{\delta \lambda_j^0}.
\end{align*}
Clearly $a^0_j \rightarrow 0$ as $\delta \rightarrow \infty$ for $j=2,\dots, n$, if $\hat{\omega}$ is bounded. But $\hat{\omega}$ must be bounded since the system matrix of \eqref{eq:cl_dynamics_vec_delta_int_distributed} is full rank, implying that the steady-state solution to \eqref{eq:cl_dynamics_vec_delta_int_distributed} is bounded. Thus, at steady-state we have $\eta_i = \eta^* \; \forall i=1, \dots, n$. Inserting this in \eqref{eq:cl_dynamics_vec_delta_int_distributed_eq_first_n} yields
\begin{align}
\frac{n k^V k^{\text{droop, I}}}{k^\omega}  \eta^* &= - 1_n^T P^m.
\label{eq:eq:cl_dynamics_vec_delta_int_distributed_eta_final}
\end{align}
Comparing \eqref{eq:eq:cl_dynamics_vec_delta_int_distributed_eta_final} with \eqref{eq:cl_dynamics_vec_delta_int_projected_equilibrium_gamma}, it is clear that the first $2n$ rows of the equilibria of \eqref{eq:cl_dynamics_vec_delta_int_projected} and \eqref{eq:cl_dynamics_vec_delta_int_distributed} are identical, and thus define the same solutions. We thus proceed only considering the equilibrium of \eqref{eq:cl_dynamics_vec_delta_int_projected}, whose last row implies
\begin{align}
\eta' &= \frac{k^\text{droop, I}}{n\gamma} 1_n^T \hat{\omega}.
\label{eq:eta'_in_omega}
\end{align}
Eliminating $\eta'$ in \eqref{eq:cl_dynamics_vec_delta_int_projected}, we obtain
\begin{align}
&{\begin{bmatrix}
{-}(k^\omega{+}k^{\text{droop}})I_n{-}\frac{k^V(k^\text{droop, I})^2}{n\gamma k^\omega}  1_{n \times n} & k^V I_n  \\
\frac{k^\omega}{V^{\text{nom}}} I_n & {-}(\mathcal{L}_R{+}\frac{k^V}{V^{\text{nom}}} I_n )
\end{bmatrix}}
\begin{bmatrix}
\hat{\omega} \\ \hat{V}
\end{bmatrix} \nonumber \\
 &=
\begin{bmatrix}
- P^m \\
0_{n}\\
0
\end{bmatrix}.
\label{eq:cl_dynamics_vec_delta_int_projected_eq}
\end{align}
Premultiplying the last $n$ rows of \eqref{eq:cl_dynamics_vec_delta_int_projected_eq} with $\frac{V^{\text{nom}}}{k^\omega} \left((k^\omega{+}k^{\text{droop}})I_n + \frac{k^V(k^\text{droop, I})^2}{n\gamma k^\omega}  1_{n \times n}\right)$ and adding to the first $n$ yields
\begin{align*}
&\Bigg(- k^V I_n +  \frac{V^{\text{nom}}}{k^\omega} \left((k^\omega{+}k^{\text{droop}})I_n + \frac{k^V(k^\text{droop, I})^2}{n\gamma k^\omega}  1_{n \times n}\right) \\ & \;\;\;\;\times \left(\mathcal{L}_R{+}\frac{k^V}{V^{\text{nom}}} I_n \right) \Bigg) \hat{V} = P^m,
\end{align*}
which after some simplification gives
\begin{align}
& \Bigg( \frac{k^V}{k^\omega} \left(k^{\text{droop}}I_n + \frac{k^V(k^\text{droop, I})^2}{n\gamma k^\omega}  1_{n \times n}\right) \nonumber  \\ & \;\;\;\; + \frac{V^\text{nom}}{k^\omega}\left( k^\omega + k^\text{droop} \right) \mathcal{L}_R \Bigg) \hat{V}  \triangleq A_1 \hat{V} = P^m. \label{eq:eigenvalues_secondary_1}
\end{align}
Write $\hat{V} = \sum_{i=1}^n a^1_i v^1_i$, where $v^1_i$ is an eigenvector of $A_1$ with the corresponding eigenvalue $\lambda^1_i$. It is easily verified that $\frac{1}{\sqrt{n}}1_n$ is an eigenvalue of $A_1$, which we denote $v^1_1$. Since $A_1$ is symmetric, its eigenvectors can be chosen to form an orthonormal basis. By premultiplying \eqref{eq:eigenvalues_secondary_1} with $(v^1_j)^T$, and keeping in mind that $A_1v^1_j = \lambda^1_j v^1_j$, we obtain
$ a_j = {(v^1_j)^TP^m}/{\lambda_j^1}$. By direct computation we obtain
\begin{align*}
\lambda^1_1 &= \frac{k^V}{k^\omega} \left((k^{\text{droop}}) + \frac{k^V(k^\text{droop, I})^2}{\gamma k^\omega}  \right),
\end{align*}
by which we conclude that $\lambda_1^1 \rightarrow \infty$ as $\gamma \rightarrow 0^+$. Thus $a_1 \rightarrow 0$ as $\gamma \rightarrow 0^+$. For $i\ge 2$ we obtain after some calculations
$
\lambda_i^1 \ge \frac{V^\text{nom}}{k^\omega}\left( k^\omega + k^\text{droop} \right) \lambda_i(\mathcal{L}_R)
$.
 Thus, we obtain the following bound on $\hat{V}$:
\begin{align*}
\lim_{\gamma \rightarrow 0}  \norm{\hat{V}}_\infty &= \norm{\sum_{i=2}^n \frac{(v^1_i)^TP^m}{\lambda_i} v^1_i }_\infty \\
&\le  \frac{k^\omega \max_i \left| P^m_i  \right| }{(k^\omega + k^\text{droop})V^\text{nom}} \sum_{i=2}^n \frac{1}{\lambda^1_i(\mathcal{L}_R)}.
\end{align*}
Premultiplying the first $n$ rows of \eqref{eq:cl_dynamics_vec_delta_int_projected_eq} with $\frac{1}{k^V}(\mathcal{L}_R{+}\frac{k^V}{V^{\text{nom}}} I_n ) $, and adding to the last $n$ rows of \eqref{eq:cl_dynamics_vec_delta_int_projected_eq} yields
\begin{align*}
&\;\;\;\;\Bigg(- \frac{k^\omega}{V^{\text{nom}}} I_n + \frac{1}{k^V} \left(\mathcal{L}_R + \frac{k^V}{V^{\text{nom}}} I_n \right) \\
& \;\;\;\; \times \left( (k^\omega{+}k^{\text{droop}})I_n + \frac{k^V(k^\text{droop, I})^2}{n\gamma k^\omega}  1_{n \times n} \right) \Bigg) \hat{\omega} \\
&=  \frac{1}{k^V} \left(\mathcal{L}_R + \frac{k^V}{V^{\text{nom}}} I_n \right) P^m.
\end{align*}
After some algebra, the following expression is obtained
\begin{align}
& \;\;\;\; \Bigg( \frac{k^\text{droop}}{V^\text{nom}} I_n +\frac{k^V (k^\text{droop})^2}{V^\text{nom} k^\omega n\gamma} 1_{n\times n} + \frac{k^\text{droop}+k^\omega}{k^V} \mathcal{L}_R \Bigg)\hat{\omega} \nonumber \\
&\triangleq A_2 \hat{\omega} = \frac{1}{k^V} \left(\mathcal{L}_R + \frac{k^V}{V^{\text{nom}}} I_n \right) P^m. \label{eq:stationarity_omega}
\end{align}
Again, write $\hat{\omega} = \sum_{i=1}^n a^2_i v^2_i$,
where $v^2_i$ is the eigenvector of $A_2$ with eigenvalue $\lambda^2_i$.
Let $\lambda^2_1$ denote the eigenvalue of $A_2$ with eigenvector $\frac{1}{\sqrt{n}}1_n$. By direct computation
\begin{align*}
\lambda^2_1 &= \frac{k^\text{droop}}{V^\text{nom}}  +\frac{k^V (k^\text{droop,I})^2}{V^\text{nom} k^\omega \gamma},
\end{align*}
and clearly $\lambda^2_1 \rightarrow \infty$ as $\gamma \rightarrow 0^+$, implying that $a^2_1 \rightarrow 0$ as $\gamma \rightarrow 0^+$. For the remaining eigenvalues, i.e., $i\ge 2$, we obtain after some calculations
\begin{align*}
\lambda^2_i \ge \frac{k^\text{droop}+k^\omega}{k^V} \lambda_i(\mathcal{L}_R).
\end{align*}
By premultiplying \eqref{eq:stationarity_omega} with $(a^2_j)^T$, we obtain
\begin{align*}
a^2_j &= \frac{ (v^2_j)^T \left( \mathcal{L}_R + \frac{k^V}{V^{\text{nom}}} I_n \right) P^m}{k^V \lambda^2_j}.
\end{align*}
Thus we obtain the following bound on $\hat{\omega}$
\begin{align*}
\lim_{\gamma \rightarrow 0}  \norm{\hat{\omega}}_\infty &=   \norm{\sum_{i=2}^n \frac{ (v^2_i)^T \left(\mathcal{L}_R + \frac{k^V}{V^{\text{nom}}} I_n \right) P^m}{k^V\lambda^2_i} v^2_i }_\infty \\
&\le \sum_{i=2}^n \norm{ \frac{ \left(\lambda_i(\mathcal{L}_R) + \frac{k^V}{V^{\text{nom}}} \right) (v^2_i)^T P^m}{(k^\text{droop}+k^\omega) \lambda_i(\mathcal{L}_R) } v^2_i }_\infty \\
&\le \frac{\max_i P^m_i}{k^\text{droop}+k^\omega} \left( (n-1) + \frac{k^V}{V^\text{nom}} \sum_{i=2}^n \frac{1}{\lambda_i(\mathcal{L}_R)} \right)
,
\end{align*}
where we have used the fact that the eigenvectors of $A_2$ are also eigenvectors of $\mathcal{L}_R$.  Finally, we consider the power output of the generation control. Note that $\lim_{\gamma\rightarrow 0} \lambda^2_1=\frac{k^V (k^\text{droop,I})^2}{V^\text{nom} k^\omega \gamma}$. Thus, when $\gamma \rightarrow 0$, we have by \eqref{eq:eta'_in_omega} that
\begin{align*}
\lim_{\gamma \rightarrow 0} P^\text{gen} &= - k^\text{droop} \hat{\omega} - \frac{k^V k^\text{droop,I}}{n k^\omega} 1_n \eta ' \\
&= -\left(  k^\text{droop} I_n + \frac{k^V (k^\text{droop,I})^2}{n^2 \gamma k^\omega} 1_{n\times n}  \right) \hat{\omega} \\
&= -\left(  k^\text{droop} I_n + \frac{k^V (k^\text{droop,I})^2}{n^2 \gamma k^\omega} 1_{n\times n}  \right) \sum_{i=1}^n a^2_i v^2_i \\
&= -\left(  k^\text{droop} I_n + \frac{k^V (k^\text{droop,I})^2}{n^2 \gamma k^\omega} 1_{n\times n}  \right) \\
& \;\;\;\; \times \Bigg( \frac{ \gamma  k^\omega  1_n^T  P^m}{n k^V (k^\text{droop,I})^2 } 1_n \\
&\;\;\;\; + \sum_{i=2}^n \frac{ (v^2_i)^T \left(\mathcal{L}_R + \frac{k^V}{V^{\text{nom}}} I_n \right) P^m}{k^V\lambda^2_i} v^2_i \Bigg).
\end{align*}
Noting that $v_1^2=\frac{1}{\sqrt{n}}$, we have that $1_n^Tv_i^2=0$ for $i \ge 2$. By letting $\gamma \rightarrow 0$, the above equation simplifies to
\begin{align*}
\lim_{\gamma \rightarrow 0}  P^\text{gen} &=  - \frac{1}{n} 1_n^T P^m 1_n \\
& \;\;\;\; - \sum_{i=2}^n \frac{ k^\text{droop} (v^2_i)^T \left(\mathcal{L}_R + \frac{k^V}{V^{\text{nom}}} I_n \right) P^m}{k^V\lambda^2_i} v^2_i.
\end{align*}
Using the previously derived lower bound on $\lambda_i^2$ for $i\ge 2$, we obtain the following bound on the generated power
\begin{align*}
&\;\;\;\; \norm{\frac{1}{n} 1_n^T P^m 1_n + P^\text{gen} }_\infty \\
&= \norm{\sum_{i=2}^n \frac{ k^\text{droop} (v^2_i)^T \left(\mathcal{L}_R + \frac{k^V}{V^{\text{nom}}} I_n \right) P^m}{k^V\lambda^2_i} v^2_i}_\infty \\
&\le \sum_{i=2}^n \norm{ \frac{ k^\text{droop}  k^V \left( \lambda_i(\mathcal{L}_R) + \frac{k^V}{V^{\text{nom}}}  \right) (v^2_i)^T P^m}{k^V (k^\text{droop} + k^\omega) \lambda_i(\mathcal{L}_R)} v^2_i}_\infty \\
&\le \frac{k^\text{droop}\max_i P^m_i}{k^\text{droop}+k^\omega} \left( (n-1) + \frac{k^V}{V^\text{nom}} \sum_{i=2}^n \frac{1}{\lambda_i(\mathcal{L}_R)} \right) \qedhere
\end{align*}
\end{proof}
\begin{remark}
The upper bounds on the AC frequency and the DC voltage errors, i.e., $e^\omega$ and $e^V$ in Theorem \ref{th:equilibrium_secondary} are lower than the corresponding bounds in Theorem \ref{th:equilibrium} given the same network and controller parameters. In particular
\begin{align*}
e^V_\text{dec} - e^V_\text{dist} &= \frac{k^\omega}{nk^\text{droop}k^V}  \left| \sum_{i=1}^n P^m_i \right| \\
e^\omega_\text{dec} - e^\omega_\text{dist} &= \frac{1}{nk^\text{droop}}  \left| \sum_{i=1}^n P^m_i \right|.
\end{align*}
 However, as the bounds are conservative, no conclusion about the actual control errors can be drawn.
\end{remark}

\section{Simulations}
\label{sec:simulations}
In this section, simulations are conducted on a test system to validate the performance of the proposed controllers. The simulation was performed in Matlab, using a dynamic phasor approach based on \cite{demiray2008} The test system is illustrated in Figure~\ref{fig:testsystem}.
\begin{figure}[htb]
  \centering
  \def\svgwidth{\columnwidth}
  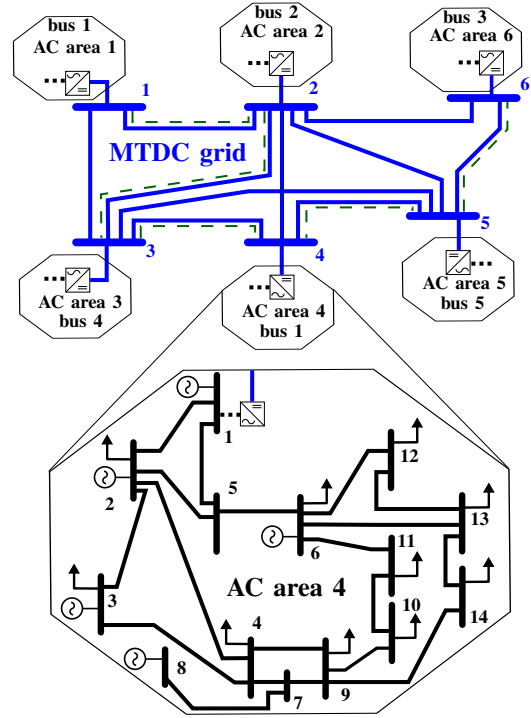
  \caption{MTDC test system, consisting of a 6-terminal MTDC grid. Each terminal is connected to an IEEE 14 bus AC grid, sketched as octagons. The dashed lines illustrate the topology of the communication grid of the controller~\eqref{eq:droop_control_secondary_distributed}.}
  \label{fig:testsystem}
\end{figure}
The parameters of the MTDC grid are given in \ref{tab:HVDCgridParameter}, and  are chosen uniformly for all VSC stations. Note that we in the simulation also consider the inductances $L_{ij}$ and capacitances $C_{ij}$ of the HVDC lines. The capacitances of the terminals are assumed to be given by $C_i=0.375\times 10^{-3}$ p.u. The AC grid parameters were obtained from \cite{Milano}. The generators are modeled as a $6^{th}$ order machine model controlled by an automatic voltage controller and a governor \cite{kundur1994power}. The loads in the grid are assumed to be equipped with an ideal power controller.
\begin{table}[htb]
\centering
\caption{HVDC grid line parameters}
\label{tab:HVDCgridParameter}
\begin{tabular}{ccccc} \toprule
$(i,j)$&$R_{ij}$ [p.u.]&$L_{ij}$ [$10^{-3}$ p.u.]& $C_{ij}$ [p.u.]\\ \midrule
(1,2), (1,3), (2,4), (3,4) & 0.0586 & 0.2560 &0.0085 \\
(2,3) & 0.0878 & 0.3840 &0.0127\\
(2,5), (4,5)  & 0.0732 & 0.3200 &0.0106\\
(2,6), (3,5), (5,6) & 0.1464 & 0.6400 &0.0212\\
  \bottomrule
\end{tabular}
\end{table}
\begin{table}[htb]
\centering
\raa{1.3}
\caption{Controller Parameter}
\label{tab:ControllerParameter}
\begin{tabular}{@{}cccccc@{}}\toprule
$K^{\omega}_i$ &$K^{V}_i$&$K^\text{droop}_i$& $K^\text{droop, I}_i$& $\gamma$  &$\delta$  \\
 \midrule
9000            &   110       & 8   & 10         &  0 & 5\\
\bottomrule
\end{tabular}
\end{table}
The controllers \eqref{eq:droop_control}, \eqref{eq:droop_control_secondary} and \eqref{eq:droop_control_secondary_distributed} were applied to the aforementioned test grid. At time $t=1$ the output of one generator in area 1 was reduced by $0.2$ p.u., simulating a fault. The communication network of controller \eqref{eq:droop_control_secondary_distributed} is illustrated by the dashed lines in Figure~\ref{fig:testsystem}.
Figure~\ref{fig:Frequ} shows the frequency response of all AC grids for the three controllers considered. Figure~\ref{fig:ResVoltage} shows the DC voltages of the terminals. Figure~\ref{fig:Generators} shows the total change in the generated power within each AC area.
It can be noted that immediately after the fault, the frequency at the corresponding AC area drops. The frequency drop is followed by a voltage drop in all terminals, and a frequency drop at all AC areas. The frequencies and voltages converge to new stationary values after approximately $30$ s. We note that the asymptotic error of the frequencies and voltages are significantly smaller when the controllers \eqref{eq:droop_control_secondary} and \eqref{eq:droop_control_secondary_distributed} are employed, than when the decentralized droop controller \eqref{eq:droop_control} is employed. The generated power is shared fairly between the AC areas.

\begin{figure}[htb]
\centering
\begin{tikzpicture}
\begin{axis}
[xlabel={$t$ [s]},
ylabel={$\omega(t)$ [p.u.]},
xmin=0,
xmax=35,
xtick={0,5,...,35},
ymin=0.9991,
ymax=1.0002,
yticklabel style={/pgf/number format/.cd,
							fixed,
                  					precision=4},
grid=major,
height=6.2cm,
width=0.95\columnwidth,
legend cell align=left,
legend pos= south east,
legend entries={Controller (3),
			    Controller (7),
			    Controller (8)},
]
\addlegendimage{no markers,black}
\addlegendimage{no markers,black,dashed}
\addlegendimage{no markers,black, dotted}

\addplot[line width=0.8pt,color=blue] table[x index=0,y index=1,col sep=tab] {Freq_NoSec.txt};
\addplot[line width=0.8pt,color=red] table[x index=0,y index=6,col sep=tab] {Freq_NoSec.txt};
\addplot[line width=0.8pt,color=green!40!black] table[x index=0,y index=11,col sep=tab] {Freq_NoSec.txt};
\addplot[line width=0.8pt,color=cyan] table[x index=0,y index=16,col sep=tab] {Freq_NoSec.txt};
\addplot[line width=0.8pt,color=magenta] table[x index=0,y index=21,col sep=tab] {Freq_NoSec.txt};
\addplot[line width=0.8pt,color=yellow!90!black] table[x index=0,y index=26,col sep=tab] {Freq_NoSec.txt};
\addplot[line width=0.8pt,color=blue,style=dashed] table[x index=0,y index=1,col sep=tab] {Freq_SecCen.txt};
\addplot[line width=0.8pt,color=red,style=dashed] table[x index=0,y index=6,col sep=tab] {Freq_SecCen.txt};
\addplot[line width=0.8pt,color=green!40!black,style=dashed] table[x index=0,y index=11,col sep=tab] {Freq_SecCen.txt};
\addplot[line width=0.8pt,color=cyan,style=dashed] table[x index=0,y index=16,col sep=tab] {Freq_SecCen.txt};
\addplot[line width=0.8pt,color=magenta,style=dashed] table[x index=0,y index=21,col sep=tab] {Freq_SecCen.txt};
\addplot[line width=0.8pt,color=yellow!90!black,style=dashed] table[x index=0,y index=26,col sep=tab] {Freq_SecCen.txt};
\addplot[line width=0.8pt,color=blue,style=dotted] table[x index=0,y index=1,col sep=tab] {Freq_SecDis.txt};
\addplot[line width=0.8pt,color=red,style=dotted] table[x index=0,y index=6,col sep=tab] {Freq_SecDis.txt};
\addplot[line width=0.8pt,color=green!40!black,style=dotted] table[x index=0,y index=11,col sep=tab] {Freq_SecDis.txt};
\addplot[line width=0.8pt,color=cyan,style=dotted] table[x index=0,y index=16,col sep=tab] {Freq_SecDis.txt};
\addplot[line width=0.8pt,color=magenta,style=dotted] table[x index=0,y index=21,col sep=tab] {Freq_SecDis.txt};
\addplot[line width=0.8pt,color=yellow!90!black,style=dotted] table[x index=0,y index=26,col sep=tab] {Freq_SecDis.txt};
\end{axis}
\end{tikzpicture}
\caption{Average frequencies in the AC areas for the controllers \eqref{eq:droop_control}, \eqref{eq:droop_control_secondary} and \eqref{eq:droop_control_secondary_distributed}, respectively.}
\label{fig:Frequ}
\end{figure}
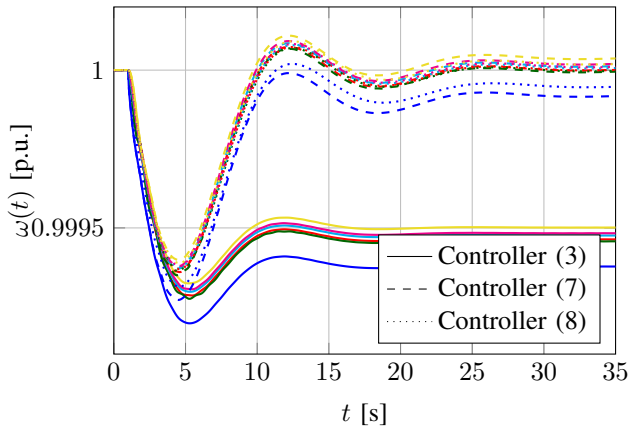

\begin{figure}[htb]
\centering
\begin{tikzpicture}
\begin{axis}[xlabel={$t$ [s]},
ylabel={$V(t)$ [p.u.]},
xmin=0,
xmax=35,
xtick={0,5,...,35},
ymin=0.93,
ymax=1.015,
yticklabel style={/pgf/number format/.cd,
                                     fixed,
                  					precision=2},
grid=major,
height=6.2cm,
width=0.975\columnwidth,
legend pos= south east,
legend cell align=left,
legend entries={Controller (3),
			    Controller (7),
			    Controller (8)},
]
\addlegendimage{no markers,black}
\addlegendimage{no markers,black,dashed}
\addlegendimage{no markers,black, dotted}
\addplot[line width=.8pt,color=blue] table[x index=0,y index=1,col sep=tab] {UDC_NoSec.txt};
\addplot[line width=.8pt,color=red] table[x index=0,y index=2,col sep=tab] {UDC_NoSec.txt};
\addplot[line width=.8pt,color=green!40!black] table[x index=0,y index=3,col sep=tab] {UDC_NoSec.txt};
\addplot[line width=.8pt,color=cyan] table[x index=0,y index=4,col sep=tab] {UDC_NoSec.txt};
\addplot[line width=.8pt,color=magenta] table[x index=0,y index=5,col sep=tab] {UDC_NoSec.txt};
\addplot[line width=.8pt,color=yellow!90!black] table[x index=0,y index=6,col sep=tab] {UDC_NoSec.txt};
\addplot[line width=.8pt,color=blue,style=dashed] table[x index=0,y index=1,col sep=tab] {UDC_SecCen.txt};
\addplot[line width=.8pt,color=red,style=dashed] table[x index=0,y index=2,col sep=tab] {UDC_SecCen.txt};
\addplot[line width=.8pt,color=green!40!black,style=dashed] table[x index=0,y index=3,col sep=tab] {UDC_SecCen.txt};
\addplot[line width=.8pt,color=cyan,style=dashed] table[x index=0,y index=4,col sep=tab] {UDC_SecCen.txt};
\addplot[line width=.8pt,color=magenta,style=dashed] table[x index=0,y index=5,col sep=tab] {UDC_SecCen.txt};
\addplot[line width=.8pt,color=yellow!90!black,style=dashed] table[x index=0,y index=6,col sep=tab] {UDC_SecCen.txt};
\addplot[line width=.8pt,color=blue,style=dotted] table[x index=0,y index=1,col sep=tab] {UDC_SecDis.txt};
\addplot[line width=.8pt,color=red,style=dotted] table[x index=0,y index=2,col sep=tab] {UDC_SecDis.txt};
\addplot[line width=.8pt,color=green!40!black,style=dotted] table[x index=0,y index=3,col sep=tab] {UDC_SecDis.txt};
\addplot[line width=.8pt,color=cyan,style=dotted] table[x index=0,y index=4,col sep=tab] {UDC_SecDis.txt};
\addplot[line width=.8pt,color=magenta,style=dotted] table[x index=0,y index=5,col sep=tab] {UDC_SecDis.txt};
\addplot[line width=.8pt,color=yellow!90!black,style=dotted] table[x index=0,y index=6,col sep=tab] {UDC_SecDis.txt};
\end{axis}
\end{tikzpicture}
\caption{DC terminal voltages for the controllers \eqref{eq:droop_control}, \eqref{eq:droop_control_secondary} and \eqref{eq:droop_control_secondary_distributed}, respectively.}
\label{fig:ResVoltage}
\end{figure}
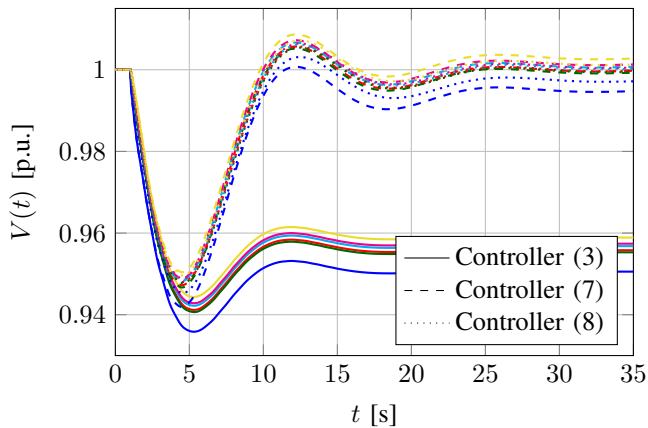

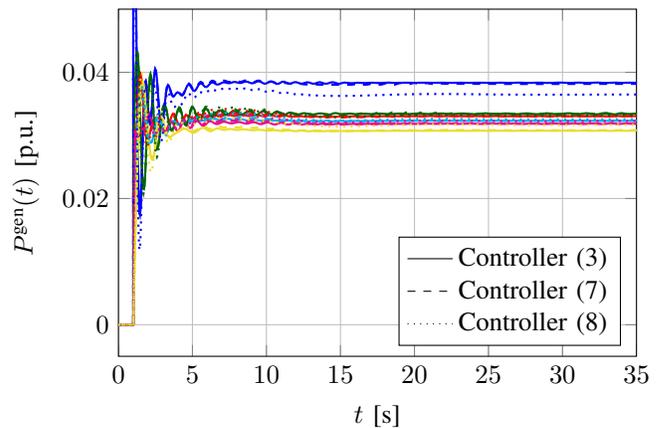
\begin{figure}[t!]
\centering
\begin{tikzpicture}
\begin{axis}
[xlabel={$t$ [s]},
ylabel={$P^\text{gen}(t)$ [p.u.]},
cycle list name=color list,
xmin=0,
xmax=35,
xtick={0,5,...,35},
ymin=-0.005,
ymax=0.05,
scaled ticks=false,
tick label style={/pgf/number format/fixed,
                  /pgf/number format/precision=2},
grid=major,
height=6.2cm,
width=0.975\columnwidth,
legend pos= south east,
legend cell align=left,
legend entries={Controller (3),
			    Controller (7),
			    Controller (8)},
]
\addlegendimage{no markers,black}
\addlegendimage{no markers,black,dashed}
\addlegendimage{no markers,black, dotted}
\addplot[line width=0.8pt,color=blue] table[x index=0,y index=1,col sep=tab] {PGen_NoSec.txt};
\addplot[line width=0.8pt,color=red] table[x index=0,y index=2,col sep=tab] {PGen_NoSec.txt};
\addplot[line width=0.8pt,color=green!40!black] table[x index=0,y index=3,col sep=tab] {PGen_NoSec.txt};
\addplot[line width=0.8pt,color=cyan] table[x index=0,y index=4,col sep=tab] {PGen_NoSec.txt};
\addplot[line width=0.8pt,color=magenta] table[x index=0,y index=5,col sep=tab] {PGen_NoSec.txt};
\addplot[line width=0.8pt,color=yellow!90!black] table[x index=0,y index=6,col sep=tab] {PGen_NoSec.txt};
\addplot[line width=0.8pt,color=blue,style=dashed] table[x index=0,y index=1,col sep=tab] {PGen_SecCen.txt};
\addplot[line width=0.8pt,color=red,style=dashed] table[x index=0,y index=2,col sep=tab] {PGen_SecCen.txt};
\addplot[line width=0.8pt,color=green!40!black,style=dashed] table[x index=0,y index=3,col sep=tab] {PGen_SecCen.txt};
\addplot[line width=0.8pt,color=cyan,style=dashed] table[x index=0,y index=4,col sep=tab] {PGen_SecCen.txt};
\addplot[line width=0.8pt,color=magenta,style=dashed] table[x index=0,y index=5,col sep=tab] {PGen_SecCen.txt};
\addplot[line width=0.8pt,color=yellow!90!black,style=dashed] table[x index=0,y index=6,col sep=tab] {PGen_SecCen.txt};
\addplot[line width=0.8pt,color=blue,style=dotted] table[x index=0,y index=1,col sep=tab] {PGen_SecDis.txt};
\addplot[line width=0.8pt,color=red,style=dotted] table[x index=0,y index=2,col sep=tab] {PGen_SecDis.txt};
\addplot[line width=0.8pt,color=green!40!black,style=dotted] table[x index=0,y index=3,col sep=tab] {PGen_SecDis.txt};
\addplot[line width=0.8pt,color=cyan,style=dotted] table[x index=0,y index=4,col sep=tab] {PGen_SecDis.txt};
\addplot[line width=0.8pt,color=magenta,style=dotted] table[x index=0,y index=5,col sep=tab] {PGen_SecDis.txt};
\addplot[line width=0.8pt,color=yellow!90!black,style=dotted] table[x index=0,y index=6,col sep=tab] {PGen_SecDis.txt};
\end{axis}
\end{tikzpicture}
\caption{Total generated power in the AC areas for the controllers \eqref{eq:droop_control}, \eqref{eq:droop_control_secondary} and \eqref{eq:droop_control_secondary_distributed}, respectively.}\label{fig:Generators}
\end{figure}

\section{Discussion and Conclusions}
\label{sec:discussion}
In this paper we have studied controllers for sharing primary and secondary frequency control reserves in asynchronous AC systems connected through an MTDC system. We have reviewed a decentralized droop controller, and later expanded this to two distributed secondary frequency controllers. The distributed controllers use both local and neighboring frequency measurements of the AC grids, as well as the local DC voltage measurements. The resulting equilibria of the closed-loop system was shown to be globally asymptotically stable by using Lyapunov arguments. We also showed bounds for the asymptotic deviations of the DC voltages and the AC frequencies from their reference values. The obtained bounds are lower than the corresponding bounds when using only decentralized droop control. Furthermore the generated power from the primary frequency control is approximately shared fairly between the AC areas, and the error from fair power sharing is bounded. We have furthermore demonstrated our results on a 6 terminal MTDC system with connected AC systems. Future work will focus on eliminating the static errors in the frequencies.

\bibliography{references}
\bibliographystyle{plain}
\end{document}